\documentclass[a4paper]{amsart}
\usepackage{amssymb, amsthm, diagrams}
\usepackage[colorlinks]{hyperref}

\renewcommand{\AA}{\mathbb A}
\newcommand{\BB}{\mathbb B}
\newcommand{\CC}{\mathbb C}
\newcommand{\DD}{\mathbb D}

\newcommand{\QQ}{\mathbb Q}
\newcommand{\RR}{\mathbb R}
\newcommand{\ZZ}{\mathbb Z}

\newcommand{\calH}{\mathcal{H}}

\newcommand{\calO}{\mathcal{O}}

\newcommand{\e}{\mathbf e}

\renewcommand{\ae}{{}_{\fa}\mathbf{e}}

\newcommand{\fa}{\mathfrak{a}}
\newcommand{\fb}{\mathfrak{b}}
\newcommand{\fc}{\mathfrak{c}}
\newcommand{\ff}{\mathfrak{f}}
\newcommand{\fg}{\mathfrak{g}}
\newcommand{\fp}{\mathfrak{p}}
\newcommand{\fpb}{{\bar{\mathfrak{p}}}}
\newcommand{\fP}{\mathfrak{P}}
\newcommand{\fq}{\mathfrak{q}}

\newcommand{\vp}{\varphi}
\newcommand{\into}{\hookrightarrow}

\newcommand{\Qp}{\QQ_p}
\newcommand{\Qpb}{\overline{\QQ}_p}
\newcommand{\Zp}{\ZZ_p}
\newcommand{\Dcris}{\DD_\mathrm{cris}}

\newcommand{\kato}{\mathbf{z}^\mathrm{Kato}}
\newcommand{\Oinf}{\widehat{\calO}_\infty}
\newcommand{\Finf}{\widehat{F}_{\infty}}

\DeclareMathOperator{\ab}{ab}

\DeclareMathOperator{\Col}{Col}
\DeclareMathOperator{\coLie}{coLie}

\DeclareMathOperator{\cris}{cris}

\DeclareMathOperator{\dR}{dR}
\DeclareMathOperator{\End}{End}

\DeclareMathOperator{\Fil}{Fil}

\DeclareMathOperator{\Gal}{Gal}

\DeclareMathOperator{\Ind}{Ind}
\DeclareMathOperator{\Iw}{Iw}

\DeclareMathOperator{\loc}{loc}

\DeclareMathOperator{\per}{per}
\DeclareMathOperator{\pr}{pr}

\DeclareMathOperator{\res}{res}

\renewcommand{\to}{\longrightarrow}

\newcommand{\fullfunction}[5]{
     \begin {array}{ccrcl}
       {#1}    & \colon  & {#2} & \longrightarrow & {#3} \\
       \mbox{} & \mbox{} & {#4} & \longmapsto     & {#5}
   \end{array}
}
\newcommand{\function}[4]{
     \begin {array}{rcl}
       {#1} & \longrightarrow & {#2} \\
       {#3} & \longmapsto     & {#4}
   \end{array}
}
\newtheorem{theorem}{Theorem}[section]
\newtheorem{proposition}[theorem]{Proposition}
\newtheorem{lemma}[theorem]{Lemma}
\newtheorem{corollary}[theorem]{Corollary}
\newtheorem{remark}[theorem]{Remark}
\newtheorem{definition}[theorem]{Definition}
\newtheorem{note}[theorem]{Note}

\begin{document}

\title{Critical slope $P$-adic $L$-functions of CM modular forms}

\author{Antonio Lei}
\address[Lei]{Department of Mathematics and Statistics\\
Burnside Hall\\
McGill University\\
Montreal QC\\
Canada H3A 2K6}
\email{antonio.lei@mcgill.ca}

\author{David Loeffler}
\address[Loeffler]{Mathematics Institute\\
Zeeman Building\\
University of Warwick\\
Coventry CV4 7AL, UK}
\email{d.a.loeffler@warwick.ac.uk}

\author{Sarah Livia Zerbes}
\address[Zerbes]{Mathematics Research Institute\\
  Harrison Building\\
  University of Exeter\\
  Exeter EX4 4QF, UK
}
\email{s.zerbes@exeter.ac.uk}

\thanks{The first author is grateful for the support of a CRM-ISM postdoctoral fellowship.}

\date{Last updated 2011-12-22}

\begin{abstract}
For ordinary modular forms, there are two constructions of a $p$-adic $L$-function attached to the non-unit root of the Hecke polynomial, which are conjectured but not known to coincide. We prove this conjecture for modular forms of CM type, by calculating the the critical-slope $L$-function arising from Kato's Euler system and comparing this with results of Bella\"iche on the critical-slope $L$-function defined using overconvergent modular symbols.
\end{abstract}

\maketitle

\section{Setup}

 \subsection{Introduction}
 
  Let $f$ be a cuspidal new modular eigenform of weight $\ge 2$, and $p$ a prime not dividing the level of $f$. It has long been known that if $\alpha$ is any root of the Hecke polynomial of $f$ at $p$ such that $v_p(\alpha) < k-1$, then there is a $p$-adic $L$-function $L_{p, \alpha}(f)$ interpolating the critical $L$-values of $f$ and its twists by Dirichlet characters of $p$-power conductor; see \cite{MSD74,amicevelu75,vishik}.
  
  If $f$ is \emph{non-ordinary} (the Hecke eigenvalue of $f$ at $p$ has valuation $> 0$) then both roots of the Hecke polynomial satisfy this condition, but if $f$ is ordinary, then there is one root with valuation $k - 1$ (``critical slope''), to which the classical modular symbol constructions do not apply. Two approaches exist to rectify this injustice to the ordinary forms by constructing a critical-slope $p$-adic $L$-function. Firstly, there is an approach using $p$-adic modular symbols \cite{pollackstevens11,pollackstevenspreprint,bellaiche11a}. Secondly, there is an approach using Kato's Euler system \cite{kato04} and Perrin-Riou's $p$-adic regulator map \cite{perrinriou95} (cf.~\cite[Remarque 9.4]{colmezBSD}). Although it is natural to conjecture that the objects arising from these two constructions coincide (cf.~\cite[Remark 9.7]{pollackstevenspreprint}), and the results of \cite{loefflerzerbes11} are strong evidence for this conjecture, prior to the present work this was not known in a single example.
  
  In this paper, we show that the two critical-slope $L$-functions coincide for modular forms of CM type. In this case, Bella\"iche has shown \cite{bellaiche11b} that the ``modular symbol'' critical-slope $p$-adic $L$-function is related to the Katz $p$-adic $L$-function for the corresponding imaginary quadratic field. We show here that the same relation holds for the Kato critical slope $p$-adic $L$-function, by comparing Kato's Euler system with another Euler system: that arising from elliptic units. Using the results of \cite{yager82} and \cite{deshalit87} relating elliptic units to Katz's $L$-function, we obtain a formula (Theorem \ref{thm:values}) for the Kato $L$-function, which coincides with Bella\"iche's formula for its modular symbol counterpart (up to a scalar factor corresponding to the choice of periods). This establishes the equality of the two critical-slope $p$-adic $L$-functions for ordinary eigenforms of CM type (Theorem \ref{thm:maintheorem}).

      
  \subsection{Notation}
  
   Let $K$ be a finite extension of either $\QQ$ or $\Qp$, where $p$ is an odd prime. We write $K_\infty=K(\mu_{p^\infty})$, $\overline{K}$ for an algebraic closure of $K$ and $K^{\ab}$ for the maximal abelian extension of $K$ in $\overline{K}$. A $p$-adic representation of the absolute Galois group $\Gal(\overline{K}\slash K)$ is a finite-dimensional $\Qp$-vector space with a continuous linear action of $\Gal(\overline{K}\slash K)$.
    
   A Galois extension $L$ of $K$ will be called a $p$-adic Lie extension if $G=\Gal(L / K)$ is a compact $p$-adic Lie group of finite dimension. In this case, we denote by $\Lambda(G)$ its Iwasawa algebra; it is defined to be the completed group ring
   \[ \Lambda(G)=\varprojlim \Zp[G\slash U],\]
   where $U$ runs over all open normal subgroups of $G$. We write $Q(G)$ for the total quotient ring of $\Lambda(G)$. If $R$ is a $p$-adically complete $\Zp$-algebra, we shall write $\Lambda_R(G)$ for $R \mathop{\widehat\otimes} \Lambda(G)$, the Iwasawa algebra with coefficients in $R$. 

   If $L$ is a complete discretely valued subfield of $\mathbb{C}_p$, we write $\calH_L(G)$ for the algebra of $L$-valued distributions on $G$ (the continuous dual of the space of locally $L$-analytic functions). This naturally contains $\Lambda_L(G)$ as a subalgebra. When $G$ is the cyclotomic Galois group $\Gamma$ (isomorphic to $\mathbb{Z}_p^\times$), and $i \in \ZZ$, we shall write $\ell_i$ for the element $\frac{\log([\gamma])}{\log \chi(\gamma)} - i$ of $\calH_{\Qp}(\Gamma)$ (where $\gamma$ is any element of $\Gamma$ of infinite order).
   
   Assume now that $K$ is a number field, and let $S$ be a finite set of places of $K$ (which we shall always assume to contain the infinite places). Let $K^S$ be the maximal extension of $K$ which is unramified outside $S$, and let $V$ be a $p$-adic representation of $\Gal(K^S / K)$. For an extension $L$ of $K$ contained in $K^S$, write $H^1_S(L,V)$ for the Galois cohomology group $H^1(\Gal(K^S\slash L),V)$. Let $T$ be a $\Gal(\overline{K}\slash K)$-stable lattice in $V$. If $L\subset K^S$ is a $p$-adic Lie extension of $K$, define
   \[ H^1_{\Iw,S}(L,T)=\varprojlim H^1_S(L_n,T),\]
   where $L_n$ is a sequence of finite Galois extensions of $K$ such that $L=\bigcup_n L_n$ and the inverse limit is taken with respect to the corestriction maps. Note that $H^1_{\Iw,S}(L,T)$ is equipped with a continuous action of $G=\Gal(L\slash K)$, which extends to an action of $\Lambda(G)$. We also define $H^1_{\Iw,S}(L,V)=H^1_{\Iw,S}(L,T)\otimes_{\Zp}\Qp$, which is independent of the choice of lattice $T$.
   
   Similarly, let $F$ be a finite extension of $\Qp$, $V$ a $p$-adic representation of $\Gal(\overline{F} / F)$ and $T$ a $\Gal(\overline{F} / F)$-invariant lattice in $V$. For a $p$-adic Lie extension $L$ of $F$ such that $L=\bigcup L_n$ with $L_n\slash F$ finite Galois, define
   \[ H^1_{\Iw}(L,T)=\varprojlim H^1(L_n,T) \hspace{5ex}\text{and} \hspace{5ex} H^1_{\Iw}(L,V)=H^1_{\Iw}(L,T)\otimes_{\Zp}\Qp.\]
   
   For a finite extension $K$ of $\QQ$, denote by $\AA_K$ the ring of ad\`eles of $K$. If $\ff$ is an integral ideal of $K$, write $K(\ff)$ for the ray class field modulo $\ff$. Let $K(\ff p^\infty)=\bigcup_n K(\ff p^n)$, and define the Galois group $G_{\ff p^\infty}=\Gal(K(\ff p^\infty)\slash K)$.


  
 \subsection{Gr\"ossencharacters}
  \label{sect:gc}
  
  Let $K$ be an imaginary quadratic field. We fix an embedding $K \into \mathbb{C}$. An algebraic Gr\"ossencharacter of $K$ of infinity-type $(m, n)$ is a continuous homomorphism $\psi: K^\times \backslash \AA_K^\times \to \CC^\times$ whose restriction to $\CC^\times$ is given by $z \mapsto z^m \bar{z}^n$. 
  
  Let $\theta$ be the Artin map $\widehat{K}^\times / K^\times \to \Gal(K^{\ab} / K)$. We choose the normalizations such that
  \[ \theta(\varpi_{\fq}) = [\fq]^{-1} \bmod I_{\fq},\]
  where $\varpi_{\fq}$ is a uniformizer at the prime $\fq$, $I_\fq$ is the inertia group and $[\fq]$ is the arithmetic Frobenius element at $\fq$. Then we have the following well-known result:
  
  \begin{theorem}[Weil, \cite{weil56}]
   Let $\psi$ be an algebraic Gr\"ossencharacter of $K$, and let $L$ be the finite extension of $\QQ$ inside $\CC$ generated by $\psi(\widehat{K}^\times)$. Then for any prime $\lambda$ of $L$, there is a (clearly unique) continuous character
   \[ \psi_{\lambda} : \Gal(\overline{K} / K) \to L_\lambda^\times\]
   with the property that
   \[ \psi_{\lambda} \circ \theta = \psi|_{\widehat{K}^\times}.\]
   The character $\psi_{\lambda}$ is unramified outside the primes dividing $\ell \ff$, where $\ell$ is the prime of $\QQ$ below $\lambda$ and $\ff$ is the conductor of $\psi$. 
  \end{theorem}
  
  The choice of normalization for the Artin map implies that 
  \[ \psi_{\lambda}([\fa]) = \psi(\fa)^{-1}\]
  for each $\fa$ coprime to $\ell \ff$. With these conventions, the Hodge--Tate weights\footnote{We adopt the convention that the cyclotomic character has Hodge--Tate weight $+1$; this is, of course, the Galois character attached to the norm map $\AA_K^\times \to \RR^\times$, which has infinity-type $(1, 1)$.} of $\psi_\lambda$ are given as follows. Let $\lambda$ be a prime of $L$, and $\mu$ a \emph{split} prime of $K$, which lie above the same prime of $L \cap K$. Then the  decomposition groups of $\mu$ and $\overline{\mu}$ in $\Gal(K^{\ab}/K)$ are each isomorphic to $\Gal(\Qp^{\ab} / \Qp)$, and the Hodge--Tate weight of $\psi_\lambda$ is $m$ at $\mu$ and $n$ at $\overline{\mu}$.
  

\section{Comparison of Euler systems}

 \subsection{Elliptic units}

  As above, let $K$ be an imaginary quadratic field, with a fixed choice of embedding $K \into \CC$. We shall fix, for the remainder of this paper, an embedding $\overline{K} \into \CC$ compatible with this choice. In particular, for each integral ideal $\ff$, we regard the ray class field $K(\ff)$ as a subfield of $\CC$, and we write $K(\ff)^+$ for its real subfield\footnote{We stress that $K(\ff)$ is not a CM field in general, so the definition of $K(\ff)^+$ depends on the choice of embedding, and in particular $K(\ff)^+$ is not a totally real field.}.
  
  \begin{definition}
   If $L$ is a subfield of $\CC$, a \emph{CM-pair} of modulus $\ff$ over $L$ is a pair $(E, \alpha)$ consisting of an elliptic curve $E / L$ and a point $\alpha \in E(L)_{\mathrm{tors}}$, such that 
   \begin{itemize}
    \item there is an isomorphism $\End_{K L}(E) \cong \calO_K$, such that the resulting action of $\End_{KL}(E)$ on $\coLie(E / KL) \cong KL$ is the natural action of $K$;
    \item the annihilator of $\alpha$ in $\calO_K$ is exactly $\ff$;
    \item there is an isomorphism $E(\CC) \to \CC / \ff$ mapping $\alpha$ to $1$.
   \end{itemize}
  \end{definition}

  Note that we do not assume that $L \supseteq K$ here, hence the slightly convoluted statement of the first condition.
  
  \begin{theorem}\label{thm:cmpair}
   Let $\ff$ be such that $\calO_K^\times \cap (1 + \ff) = \{1\}$, $\overline{\ff} = \ff$, and the smallest integer in $\ff$ is $\ge 5$. Then there exists a CM-pair of modulus $\ff$ over $K(\ff)^+$, and for any field $L$ containing $K(\ff)^+$, this CM-pair is the unique CM-pair of modulus $\ff$ over $L$ up to unique isomorphism.
  \end{theorem}

  \begin{proof}
   Consider the canonical CM-pair $(\CC / \ff, 1)$ over $\CC$. This corresponds to a point $P_\ff$ on the modular curve $Y_1(N)(\mathbb{C})$, where $N$ is the smallest integer in $\ff$.
   
   Since $N \ge 5$ by assumption, the curve $Y_1(N)$ has a canonical model over $\QQ$ such that $Y_1(N)(L)$ parametrises elliptic curves over $L$ with a point of order $N$ for each $L \subseteq \CC$. Our claim is then precisely that $P_\ff \in Y_1(N)(K(\ff)^+)$. 
   
   It is clear that $P_\ff \in Y_1(N)(\RR)$, since there is a canonical isomorphism from $\CC / \ff$ to the elliptic curve $E_{\RR}  = \{ y^2 = 4x^3 - g_2 x - g_3\}$ where $g_2$ and $g_3$ are the usual weight 4 and 6 Eisenstein series, given by $z \mapsto (\wp(z, \ff), \wp'(z, \ff))$. Since $\ff = \bar\ff$, the coefficients $g_2$ and $g_3$ are real, so $E_{\RR}$ is indeed defined over $\RR$; and as $\overline{\wp(z, \Lambda)} = \wp(\bar z, \bar \Lambda)$, this uniformization maps $1 \in \CC / \ff$ to a real point of $E_{\RR}$. Hence $P_{\ff} \in Y_1(N)(\RR)$. 
   
   On the other hand, it is well known that there exists a CM-pair of modulus $\ff$ over $K(\ff)$ (whether or not $\bar\ff = \ff$), so $P_\ff \in Y_1(N)(K(\ff))$. Hence $P_\ff \in Y_1(N)(K(\ff)^+)$.
  \end{proof}

  \begin{remark}
   It follows from this construction that the canonical CM pair $(E, \alpha)$ over $K(\ff)^+$ becomes isomorphic over $\RR$ to $(E_{\RR}, \text{image of $1 \in \CC$})$. So the complex conjugation automorphism of $E(\CC)$ arising from this $K(\ff)^+$-model corresponds to the natural complex conjugation on $\CC / \ff$.
  \end{remark}
  
  We recall the theory of elliptic units, as described in \cite[\S 15.5-6]{kato04}.
  
  \begin{theorem}
   For each pair $(\ff, \fa)$ of ideals of $K$ such that $\calO_K^\times \cap (1 + \ff) = \{1\}$ and $\fa$ is coprime to $6\ff$, there is a canonical element
   \[ \ae_{\ff} \in K(\ff)^\times,\]
   the \emph{elliptic unit} of modulus $\ff$ and twist $\fa$. If $\ff$ has at least two prime factors, $ \ae_{\ff} \in \calO_{K(\ff)}^\times$; and for any two ideals $\fa, \fb$ coprime to $6\ff$, we have
   \[ (N(\fb) - [\fb]) \cdot \ae_{\ff} = (N(\fa) - [\fa]) \cdot {}_\fb \e_{\ff},\]
   where $[\fa] = \left(\frac{\fa}{K(\ff) / K}\right) \in \Gal(K(\ff) / K)$ is the arithmetic Frobenius element at $\fa$.
  \end{theorem}

  Vital for our purposes is the following complex conjugation symmetry of the elliptic units:

  \begin{proposition}\label{prop:ellunitsreal}
   If $\ff$ satisfies the hypotheses of Theorem \ref{thm:cmpair}, then we have
   \[ \overline{\ae_\ff} = {}_{\bar\fa}\mathbf{e}_{\ff}.\]
  \end{proposition}
  
  \begin{proof}
   This follows from the construction of the elliptic units. We have 
   \[ \ae_\ff = {}_\fa \theta_E(\alpha)^{-1}\]
   where $(E, \alpha)$ is the canonical CM pair over $K(\ff)$, and ${}_\fa \theta_E$ is the element of the function field of $E$ constructed in \cite[\S 15.4]{kato04}. 
   
   By Theorem \ref{thm:cmpair}, $E$ admits a model over $K(\ff)^+$, and it is clear that if $\iota$ is the nontrivial element of $\Gal(K(\ff) / K(\ff)^+)$ arising from complex conjugation, we have $\iota({}_{\fa}E) = {}_{\bar\fa}E$ and hence (by the uniqueness of ${}_\fa \theta_E$) we have $({}_\fa \theta_E)^\iota = {}_{\bar\fa} \theta_E$. Since $\alpha \in E(K(\ff)^+)$, we deduce that
   \[ \overline{\ae_\ff} = ({}_\fa \theta_E)^\iota(\alpha)^{-1} = {}_{\bar\fa} \theta_E(\alpha)^{-1} = {}_{\bar\fa}\mathbf{e}_{\ff}\]
   as required.
  \end{proof}

  \begin{remark}
   Modulo differing choices of conventions, this is the formula labelled ``Transport of Structure'' in \S 2.5 of \cite{gross80}.
  \end{remark}

%
%
%

 \subsection{Elliptic units in Iwasawa cohomology}
  
  Let $p$ be a rational prime which splits in $K$. For fixed $\ff$ (which we shall assume prime to $p$), the ideal $\fg = \ff p^n$ satisfies the condition $\calO_K^\times \cap (1 + \fg) = \{1\}$ for all $n \gg 0$, so if $(\fa, 6p\ff) = 1$ we may define the elements $\ae_{\ff p^n}$. These are \emph{norm-compatible} (c.f. \cite[\S15.5]{kato04}), and we may extend their definition to all $n \ge 0$ using the norm maps. 
  
  \begin{note}
   Since $\ff p^n$ has at least two prime factors for $n \ge 1$, we have $\ae_{\ff p^n} \in \calO_{K(\ff p^n)}^\times$.
  \end{note} 
  
  Let $S$ be a set of places of $K$ containing the infinite places and the primes above $p$. Then we have the Kummer maps
  \[ \kappa_L : \Zp \otimes_{\ZZ} \calO_{L, S}^\times \rTo^\cong H^1_S(L, \Zp(1)).\]
  Since the sequence of elements $\ae_{\ff p^\infty} = (\ae_{\ff p^n})_{n \ge 0}$ is a norm-compatible sequence of units, their images under the Kummer maps are corestriction-compatible, so we obtain an element
  \[ \ae_{\ff p^\infty} \in H^1_{\Iw, S} (K(\ff p^\infty) , \Zp(1)) = \varprojlim_{n} H^1_S(K(\ff p^n), \Zp(1)).\]
    
  \begin{theorem}\label{thm:ellunitsreal2}
   If $\ff$ is Galois-stable, then we have 
   \[ \iota_* \left(\ae_{\ff p^\infty}\right) = {}_{\bar\fa}\e_{\ff p^\infty},\]
   where $\iota_*$ is the involution of $H^1_{\Iw, S} (K(\ff p^\infty), \Zp(1))$ induced by complex conjugation.
  \end{theorem}

  \begin{proof} 
   Immediate from Proposition \ref{prop:ellunitsreal}, since $\ff p^n$ satisfies the conditions of Theorem \ref{thm:cmpair} for all $n \gg 0$.
  \end{proof}  
  
  \begin{definition}
   We also define the element
   \[ \e_{\ff p^\infty} = (N(\fa) - [\fa])^{-1} \cdot \ae_{\ff p^\infty} \in Q(G_{\ff p^\infty}) \otimes_{\Lambda(G_{\ff p^\infty})} H^1_{\Iw, S} (K(\ff p^\infty), \Zp(1)),\]
   where $\Lambda(G_{\ff p^\infty})$ is the Iwasawa algebra of $G_{\ff p^\infty} = \Gal(K(\ff p^\infty) / K)$ and $Q(G_{\ff p^\infty})$ its total ring of quotients. 
  \end{definition}
  
  \begin{note}
   The element $\e_{\ff p^\infty}$ is independent of the choice of $\fa$. 
  \end{note} 
  
  \begin{corollary}\label{cor:iotainvariant}
   We have $\iota_*(\e_{\ff p^\infty}) = \e_{\ff p^\infty}$.
  \end{corollary}
  \begin{proof}
   The automorphism $\iota_*$ of $H^1_{\Iw, S} (K(\ff p^\infty), \Zp(1))$ is $\Lambda(G_{\ff p^\infty})$-semilinear, with the action of $\iota$ on $G_{\ff p^\infty}$ being given by conjugation in $\Gal(\overline{K} / \QQ)$; hence $\iota_*$ extends canonically to the tensor product with $Q(G_{\ff p^\infty})$; and since $\iota [\fa]\iota = [\bar\fa]$, this finishes the proof by Theorem \ref{thm:ellunitsreal2}.
  \end{proof}
    
  Let $W$ be any continuous representation of $G_{\ff p^\infty}$ on a one-dimensional vector space over some finite extension $L$ of $\Qp$. Then we have an isomorphism
  \begin{equation}\label{eq:twistwithW} 
   H^1_{\Iw, S}(K(\ff p^\infty), \Zp(1)) \otimes_{\Zp} W \rTo^\cong H^1_{\Iw, S}(K(\ff p^\infty), W(1)).
  \end{equation}
     
  \begin{definition}
   For an element $w \in W$,  let $\e_{\ff p^\infty}(w)$ be the image of $\e_{\ff p^\infty} \otimes w$ under \eqref{eq:twistwithW}, which is an element of
  \[ Q(G_{\ff p^\infty}) \otimes_{\Lambda(G_{\ff p^\infty})} H^1_{\Iw, S}(K(\ff p^\infty), W(1)).\]
  Define
  \[ \e_\infty(w) \in Q(\Gamma) \otimes_{\Lambda(\Gamma)} H^1_{\Iw, S}(K_\infty, W(1))\]
  to be the image of $\e_{\ff p^\infty}(w)$ under the corestriction map 
  \[ H^1_{\Iw, S}(K(\ff p^\infty), W(1)) \rTo H^1_{\Iw, S}(K_\infty, W(1)).\]
  \end{definition}
  
  \begin{lemma}
   If $W$ has no fixed points under $\Gal(K(\ff p^\infty) / K_\infty)$, then we have
    \[ \e_{\infty}(w) \in H^1_{\Iw, S}(K_\infty, W(1)).\]
  \end{lemma}
  
  \begin{proof}
    Let $I$ be the ideal in $\Lambda(\ff p^\infty)$ generated by the elements $(N\fa - [\fa])$ for integral ideals $\fa$ prime to $6\ff$. Suppose $G_{\ff p^\infty}$ acts on $W$ via the character $\tau: G_{\ff p^\infty} \to L$. Then we must show that the ideal in $\Lambda(\Gamma)$ generated by the elements
   \[ \{ (N\fa - \tau([\fa])^{-1} [\fa]) : \fa \text{ is an integral ideal coprime to $6\ff$}\}\]
   contains a power of $p$. However, if this is not the case, it must consist of elements of $\Lambda(\Gamma)$ which all vanish at some character $\eta$ of $\Gamma$. Then $\chi([\fa]) \tau([\fa]) - \eta([\fa])$ vanishes for every $\fa$. By the Chebotarev density theorem, we must have $\tau = \chi^{-1} \eta$, which contradicts the assumption that $\tau$ does not factor through $\Gamma$.
  \end{proof}

  We write $\iota W$ for the representation of $G_{\ff p^\infty}$ that acts on $\{\iota w:w\in W\}$ via $g\cdot(\iota w)=\iota (\iota g\iota)\cdot w$.
 
  \begin{theorem}\label{thm:invariant}
   If $W$ has no fixed points under $\Gal(K(\ff p^\infty) / K_\infty)$, the element
    \[ \e_\infty(w) \in H^1_{\Iw, S}(K_\infty / K, W(1))\]
    satisfies
    \[ \iota_*(\e_\infty(w)) = \e_\infty(\iota w)\]
   where $\iota_*$ is induced from the maps
    \[ H^1_S(K(\ff p^n), W(1)) \rTo H^1_{S}(K(\ff p^n), (\iota W)(1))\]
   sending a cocycle $\tau$ to the cocycle $g \mapsto \iota \tau(\iota g \iota)$, for each $n \ge 0$.
  \end{theorem}

  We split the proof of the theorem into a number of steps. 
  
   \begin{definition}   
    Let $\Lambda^{\sharp}(G_{\ff p^\infty})(1)$ denote $\Lambda(G_{\ff p^\infty})(1)$ endowed with the action of $\Gal(K^S / K)$ via the product of the cyclotomic character with the \emph{inverse} of the canonical character $\Gal(K^S / K) \twoheadrightarrow G_{\ff p^\infty} \into \Lambda(G_{\ff p^\infty})^\times$, i.e. $g.\omega=\chi(g) \bar g^{-1}\omega$ for any $g\in \Gal(K^S / K) $ and $\omega\in \Lambda^\sharp(G)$. Here, $\bar g$ denotes the image of $g$ in $G_{\ff p^\infty}$.    
   \end{definition}
    
  \begin{lemma}\label{lem:twistdiagram}
   We have a commutative diagram
   \begin{equation}\label{iotatwist}
    \begin{diagram}
     H^1_{\Iw,S}(K(\ff p^\infty), \Zp(1)) \otimes_{\Zp} W & \rTo^\cong & H^1_{\Iw,S}(K(\ff p^\infty), W(1))\\
     \dTo^{\iota_* \otimes \iota} & & \dTo^{\iota_*}\\
     H^1_{\Iw,S}(K(\ff p^\infty), \Zp(1)) \otimes_{\Zp} \iota W &\rTo^\cong& H^1_{\Iw,S}(K(\ff p^\infty), (\iota W)(1))\\
    \end{diagram}
    \end{equation}
   where the left-hand vertical map is the tensor product of the automorphism $\iota_*$ of $H^1_{\Iw, S}(K_\infty, \Zp(1))$ and the canonical map $\iota: W \to \iota W$, and the right-hand vertical map is as defined in the statement of Theorem~\ref{thm:invariant}.
  \end{lemma} 
  
  \begin{proof}
   We will deduce this isomorphism by using an alternative definition of the Iwasawa cohomology which renders the horizontal maps in the diagram easier to handle.  By Shapiro's lemma, we have a canonical isomorphism of $\Lambda(G_{\ff p^\infty})$-modules
   \[ H^1_{\Iw, S}(K(\ff p^\infty), M(1)) \cong H^1_{S}(K, M \otimes_{\Zp} \Lambda^{\sharp}\big(G_{\ff p^\infty})(1)\big)\]
   for any $\Gal(K^S / K)$-module $M$ which is finite-rank over $\Zp$ or $\Qp$.
   
   Let $\tau$ be the character by which $G_{\ff p^\infty}$ acts on $W$, and define $\tau_*:\Lambda^\sharp(G)\rightarrow \Lambda^\sharp(G)$ to be the map induced by $g\rightarrow \tau(g)^{-1}g$. Then the natural twisting map 
   \[ j:H^1_S\big(K,\Lambda^\sharp(G)(1)\big)\otimes W \rTo^\cong H^1_S\big(K,\Lambda^\sharp(G)(1)\otimes W\big),\]
   is explicitly given as follows: if $c:\Gal(K^S\slash K)\rightarrow \Lambda^\sharp(G)(1)$ is a cocycle and $w\in W$, define
   \[ j(c\otimes w)(g)=\tau_*(c(g))\otimes w.\] 
   We check that $j(c\otimes w)$ is a cocycle. Let $h,g\in \Gal(K^S\slash K)$. Then 
   \begin{align*}
    j(c\otimes w)(gh) & = \tau_*(c(gh))\otimes w \\
                      & = \tau_*(g.c(h))\otimes w+\tau_*c(g)\otimes w \\
                      & = \chi(g)\tau_*(g^{-1}c(h))\otimes w+\tau_*c(g)\otimes w \\
                      & = \chi(g)\tau(g)\hspace{1ex}g^{-1}[\tau_*(c(h))]\otimes w+ \tau_*(c(g))\otimes w \\
                      & = g.[j(c\otimes w)(h)]+ j(c\otimes w)(g)
   \end{align*}
   Rewrite the diagram \eqref{iotatwist} as
   \begin{equation}
    \begin{diagram}
    H^1_{S}(K, \Lambda^\sharp(G)(1)) \otimes_{\Zp} W & \rTo^{j_W} & H^1_{S}(K, \Lambda^\sharp(G)(1)\otimes W)\\
    \dTo^{\iota_* \otimes \iota} & & \dTo^{\iota_*}\\
    H^1_{S}(K, \Lambda^\sharp(G)(1)) \otimes_{\Zp} \iota W & \rTo^{j_{\iota W}} & H^1_{S}(K, \Lambda^\sharp(G)(1)\otimes \iota W)\\
   \end{diagram}
   \end{equation}
   It is then immediate from the description of $j$ that the diagram commutes, which finishes the proof.
  \end{proof}
  
  \begin{proof}[Proof of Theorem \ref{thm:invariant}]
   By Corollary \ref{cor:iotainvariant} and Lemma \ref{lem:twistdiagram}, we have
   \[ \iota_*(\e_{\ff p^\infty}(w))= \e_{\ff p^\infty}(\iota w).\]
   The action of $\iota_*$ is clearly compatible with corestriction, so we have a commutative diagram
    \begin{diagram}
     H^1_{\Iw,S}(K(\ff p^\infty), W(1)) & \rTo & H^1_{\Iw,S}(K_\infty, W(1))\\
    \dTo^{\iota^*} & & \dTo^{\iota_*}\\
     H^1_{\Iw,S}(K(\ff p^\infty), (\iota W)(1)) & \rTo & H^1_{\Iw,S}(K_\infty, \iota W(1))\\
    \end{diagram}
    which implies that $\iota_*(\e_\infty(w))= \e_\infty(\iota w)$, completing the proof.
  \end{proof}
  
  \begin{lemma}\label{lem:restrictionmap}
   Let $V$ be any $p$-adic representation of $\Gal(K^S / \QQ)$. Then the restriction map induces an isomorphism
   \[ H^1_{\Iw, S}(\QQ_\infty, V) \rTo H^1_{\Iw, S}(K_\infty, V)^{\Gal(K_\infty\slash\QQ_\infty)}. \]
  \end{lemma}
  
  \begin{proof}
   The restriction map is induced from the restriction maps on finite level, which fit into the exact sequence
   \begin{multline*} 
    0 \rTo H^1\big(\Gal(K_n\slash \QQ_n), V^{\Gal(K^S\slash K_n)}\big) \rTo H^1_{S}(\QQ_n, V) \\\rTo H^1_{S}(K_n, V)^{\Gal(K_n\slash\QQ_n)} 
        \rTo H^2\big(\Gal(K_n\slash \QQ_n), V^{\Gal(K^S\slash K_n)}\big). 
   \end{multline*}
   Since $\Qp$ has characteristic 0, the higher cohomology groups of any $\Qp$-linear representation of the cyclic group of order 2 are zero. This gives the claim at each finite level, and hence in the inverse limit.
  \end{proof} 
   
   Let $\alpha$ be the unique nontrivial element of $\Gal(K_\infty / \QQ_\infty)$.
   
   \begin{lemma}\label{lem:alphaandiota}
    We have $\alpha= \delta\iota$, where $\delta$ is the unique element of $\Gal(K_\infty\slash K)$ which acts on $\QQ_\infty$ as complex conjugation. In particular, $\delta$ is of order $2$. 
   \end{lemma}
   
  \begin{corollary}\label{cor:alphainvariant}
   If $\alpha$ is the unique nontrivial element of $\Gal(K_\infty / \QQ_\infty)$, then for any $w \in W$,
   \[ \alpha_*\left(\e_\infty(w)\right) = \delta \cdot \e_{\infty}(\iota w).\]
  \end{corollary}
  
  \begin{proof}\relax
   As above, write $\alpha= \delta\iota$. By Lemma \ref{lem:restrictionmap}, we have $\iota^* \cdot e_{\infty}(w) = \e_{\infty}(\iota w)$. Hence $\alpha_*\left(\e_\infty(w)\right) = \delta \cdot \iota_*\left(\e_\infty(w)\right) = \delta \cdot  \e_{\infty}(\iota w)$.
  \end{proof}
    

 \subsection{The two-variable \texorpdfstring{$L$}{L}-function of \texorpdfstring{$K$}{K}}
 
  We recall the construction (originally due to Yager \cite{yager82}) of a two-variable $p$-adic $L$-function from the elliptic units.
  
  Let $\fp$ be one of the two primes of $K$ above $p$. We choose an embedding $\overline{K} \into \Qpb$ inducing the $\fp$-adic valuation on $K$. Then for any finite extension $L / K$, and any $\Gal(\overline{K} / K)$-module $M$, we may define 
  \[ Z^1_\fp(L, M) = \bigoplus_{\fq \mid \fp} H^1(L_\fq, M) = H^1(K_\fp, \Ind_{L}^K M).\]
  which is a $\Gal(L / K)$-module. We also define
  \[ Z^1_{\Iw,\fp}(K(\ff p^\infty), M) = \varprojlim_{L} Z^1_{\fp}(L, M)\]
  where the limit is taken over finite extensions $L / K$ contained in $K(\ff p^\infty)$.
  
  We now recall the theory of two-variable Coleman series, as introduced, under certain additional hypotheses, by Yager \cite{yager82}, and generalized to the semi-local situation here by de Shalit \cite[\S II.4.6]{deshalit87}. Let $\zeta = (\zeta_{p^n})_{n \ge 0}$ be a compatible system of $p$-power roots of unity in $\overline{K}$; and let $\Finf$ be the completion of $K(\ff \fpb^\infty)$ with respect to the prime $\fP$ of $\overline{K}$ above $\fp$ induced by our choice of embedding $\overline{K} \into \Qpb$, and $\Oinf$ the ring of integers of $\Finf$. (Thus $\Oinf$ is a complete discrete valuation ring with maximal ideal generated by $p$, and its residue field is a finite extension of the unique $\Zp$-extension of $\mathbb{F}_p$.) 
  
  \begin{proposition}\label{prop:yagercolemanmap}
   There is a unique morphism of $\Lambda(G_{\ff p^\infty})$-modules
   \[ \Col^\zeta: Z^1_{\Iw,\fp}(K(\ff p^\infty), \Zp(1)) \to \Lambda_{\Oinf}(G_{\ff p^\infty})\]
   with the following property:
  
   For each finite-order character $\eta$ of $G_{\ff p^\infty}$ which is not unramified at $\fp$, we have
   \[ \Col^\zeta(u)(\eta) = \tau(\eta, \zeta)^{-1} \eta(\tilde\vp)^n \left(\sum_{\sigma \in G_{\ff p^m}} \eta(\sigma)^{-1} \log_\fP(u_m^\sigma) \right).\]
   Here $\tilde\vp$ is the unique lifting of the arithmetic Frobenius of $\Gal(K(\ff \fpb^\infty) / K)$ to $\Gal(K(\ff p^\infty) / K_\infty)$, $m$ is any integer such that $\eta$ factors through the quotient $G_{\ff p^m} = \Gal(K(\ff p^m) / K)$, $\log_{\fP}$ is the logarithm map
   \[ \calO_{K(\ff p^n), \fP}^\times \rTo K(\ff p^n)_\fP,\]
   and 
   \[ \tau(\eta, \zeta) = \sum_{\sigma \in \Gal(K(\ff \fpb^\infty)(\mu_{p^n}) / K(\ff \fpb^\infty))} \omega(\sigma)^{-1} \zeta_{p^n}^\sigma,\]
   where $n$ is the exact power of $\fp$ dividing the conductor of $\eta$.
  \end{proposition}
  
  \begin{definition}
   We let
   \[ \mathbb{L}_{\ff p^\infty} = \Col^\zeta(\e_{\ff p^\infty}) \in \Oinf \mathop{\widehat{\otimes}_{\Zp}} Q(G_{\ff p^\infty}).\]
  \end{definition}

  \begin{proposition}
   The element $\mathbb{L}_{\ff p^\infty}$ lies in $\Lambda_{\Oinf}(G_{\ff p^\infty})$, and it coincides with the measure $\mu(\ff \fpb^\infty)$ in \cite[Theorem II.4.14]{deshalit87}.
  \end{proposition}
  
  \begin{proof}
   We have $(N\fa - [\fa]) \cdot \mathbb{L}_{\ff p^\infty} \in \Lambda_{\Oinf}(G_{\ff p^\infty})$ for all $\fa$. Since the ideal generated by $N\fa - [\fa]$ for all integral ideals $\fa$ coprime to $6\ff$ has height 2, this implies that $\mathbb{L}_{\ff p^\infty} \in \Lambda_{\Oinf}(G_{\ff p^\infty})$ (cf.~\cite[\S II.4.12]{deshalit87}).
   
   To show that the resulting measure coincides with de Shalit's $\mu(\ff \fpb^\infty)$, we compare the defining property of the map $\Col$ above with \cite[Theorem II.5.2]{deshalit87}. For a finite-order character $\eta$ of $G_{\ff p^n}$, whose conductor $\fg$ is divisible by $\fp$ and satisfies $\calO_K^\times \cap (1 + \fg) = \{1\}$, de Shalit shows that
   \[ \eta(\mu(\ff \fpb^\infty)) = \frac{-1}{12 g} G(\eta) \sum_{\fc \in \operatorname{Cl}(\fg)} \eta^{-1}([\fc]) \log \phi_{\fg}(\fc),\]
   where $g$ is the smallest rational integer in $\fg$, $\phi_{\fg}(\fc)$ is Robert's invariant and the quantity $G(\eta)$ coincides with what we have called $\tau(\eta, \zeta)^{-1} \eta(\tilde\vp)^n$. Since
   \[   (N(\fa) - [\fa])\phi_{\fg}(\fc) = [\fc] \cdot \left(\ae_\fg\right)^{-12g},\]
   this shows that the two measures coincide at every finite-order character, and hence they are equal in $\Lambda_{\Oinf}(G_{\ff p^\infty})$.
  \end{proof}
  
  \begin{note}
   If one identifies $G(\ff p^\infty)$ with the ray class group modulo $\ff p^\infty$ via the Artin map, normalized as in \S \ref{sect:gc} above, then this measure coincides with the pullback of the Katz two-variable $L$-function of $K$ (cf.~\cite[\S 4]{hidatilouine93}) up to a difference of signs. This remark will be important in the proof of Theorem \ref{thm:maintheorem} below.
  \end{note}

 \subsection{Kato's zeta element}

  Let $f=\sum a_n q^n$ be a modular form of CM type, corresponding to a Gr\"ossencharacter $\psi$ of $K$ with infinity-type $(1-k, 0)$ where $k$ is the weight of $f$. It is clear that the coefficient field $F=\QQ(a_n:n\ge 1)$ of $f$ is contained in the finite extension $L / K$ contained in $\CC$ generated by $\psi(\widehat{K}^\times)$.
  
  Following \cite[\S6.3]{kato04}, we write $S(f)$ and $V(f)$ for the subspaces of the de Rham and Betti cohomology of the Kuga--Sato variety attached to $f$. Note that both of these are $F$-vector spaces, and $S(f)$ is $1$-dimensional over $F$ while $V(f)$ is 2-dimensional. For a commutative ring $A$ over $F$, define $S_A(f)=S(f)\otimes_{F}A$ and $V_A(f)=V(f)\otimes_{F}A$. If $\lambda$ is a place of $F$ above $p$, we may identify $V_{F_\lambda}(f)$ with the $p$-adic representation associated to $f$ of Deligne \cite{deligne69} and $S_{F_\lambda}(f)$ may be identified with $\Fil^1\Dcris(V_{F_\lambda}(f))$.

  \begin{definition}
   Let $\chi$ be a Dirichlet character of conductor $p^n$. We define the maps $\theta_{\chi,f}^\pm$ by
   \[
   \fullfunction{\theta_{\chi,f}^\pm}
    {S(f)\otimes_\QQ\QQ(\mu_{p^n})} 
    {V_{\CC}(f)^\pm} 
    {x\otimes y}
    {\sum_{\sigma\in G_n}\chi(\sigma)\sigma(y)\per_{f}(x)^\pm}
   \]
   where $G_n=\Gal(\QQ(\mu_{p^n})/\QQ)$, $\per_f:S(f)\rTo V_\CC(f)$ is the period map as defined in \cite[\S6.3]{kato04} and $\gamma\mapsto\gamma^\pm$ is the projection from $V_{\CC}(f)$ to its ($1$-dimensional) $\pm1$-eigenspace for the complex conjugation.
  \end{definition}

  \begin{theorem}[{\cite[Theorem~12.5(1)]{kato04}}]\label{thm:interpolationkato}
   We have a $L_\lambda$-linear map 
   \[
    \function{V_{L_\lambda}(f)}{H^1_{\Iw,S}(\QQ_\infty,V_\lambda(f))}{\gamma}{\kato_{\gamma}}
   \]
   which satisfies the following. Let $\chi$ be a Dirichlet character of conductor $p^n$, $\gamma\in V_{L}(f)$ and $1\le r \le k-1$, then
   \[
    \theta_{\chi,f}^\pm\circ\exp^*\left(\kato_\gamma\otimes(\zeta_{p^n})^{\otimes(k-r)}\right)=(2\pi i)^{k-r-1} L_{\{p\}}(f^*,\chi,r)\cdot\gamma^\pm
   \]
   where $\pm=(-1)^{k-r-1}\chi(-1)$.
  \end{theorem}
  
  Let $\ff$ be an ideal of $\calO_K$ satisfying the conditions in Theorem~\ref{thm:cmpair} which is contained in the conductor of $\psi$. Let $(E,\alpha)$ be the canonical CM-pair over $K(\ff)$. Following \cite[\S15.8]{kato04}, we define $V_L(\psi)=H^1(E(\CC),\QQ)^{\otimes(k-1)}\otimes_K L$ and $S(\psi)=H^0(\Gal(K(\ff)/K),\coLie(E)^{\otimes(k-1)}\otimes_K L)$, where the action of $\Gal(K(\ff)/K)$ on the space $\coLie(E)^{\otimes(k-1)}\otimes_K L$ is as described in \emph{op.cit.}. Both of these are $1$-dimensional $L$-vector spaces. For any commutative ring $A$ over $L$, we write $V_A(\psi)=V_L(\psi)\otimes_LA$ and $S_A(\psi)=S(\psi)\otimes_LA$. The Galois group $\Gal(\overline{K}/K)$ acts on $V_L(\psi)\otimes_{L}L_{\lambda}$ via $\psi_\lambda$, and there exists a period map 
  \[
   \per_\psi:S(\psi)\rTo V_{\CC}(\psi)
  \]
  induced by passing to the $(k-1)$-st tensor power from the comparison isomorphism $\per_\infty$ described above.
  
  We now recall Kato's results on the relation between this zeta element and the elliptic units.

  \begin{lemma}[{\cite[Lemma 15.11]{kato04}}]\label{lem:katoisomorphisms} Fix a choice of isomorphism of $L$-vector spaces
  \[ s: S(\psi)\rTo^\sim S_L(f).\]
   \begin{itemize}
    \item[(a)] There exists a unique isomorphism of representations of $\Gal(\overline{\QQ}/\QQ)$ over $L_\lambda$
    \[
     \widetilde{V_{L_\lambda}(\psi)}\rTo V_{L_\lambda}(f)
    \]
    such that the isomorphism $S_{L_\lambda}(\psi)\rTo S_{L_\lambda}(f)$ induced by the functoriality of $\DD_{\dR}$ is compatible with $s$.
    \item[(b)] There exists a unique isomorphism of representations of $\Gal(\CC/\RR)$ over $L$
    \[
     \widetilde{V_{L}(\psi)}\rTo V_{L}(f)
    \]
    for which the diagram
    \[
     \begin{diagram}
      S(\psi) & \rTo^{\per_\psi} & \widetilde{V_{\CC}(\psi)}\\
      \dTo & & \dTo\\
      S_L(f) & \rTo^{\per_{f}} & V_{\CC}(f)
     \end{diagram}
    \]
    commutes. 
   \end{itemize}
  \end{lemma}
  
  Note that the isomorphism of part (b) implies an isomorphism $V_{L_\lambda}(\psi) \stackrel{\cong}{\to} V_{L_\lambda}(f)$ on extending scalars to $L_\lambda$, but one does not know that this coincides with the isomorphism of part (a), as remarked in \cite[\S 15.11]{kato04}.
  
  \begin{definition}
   We write $\Phi_{\psi,f}$ for the canonical map
   \[
    H^1_{\Iw,S}(K(\ff p^\infty),V_{L_\lambda}(\psi))\rTo H^1_{\Iw,S}(\QQ_\infty,V_{L_\lambda}(f))
   \]
   as defined in \cite[(15.12.1)]{kato04}.
  \end{definition}
  
  Concretely, this map can be defined as follows:
  \[
   \begin{split}
    H^1_{\Iw,S}(K(\ff p^\infty),V_{L_\lambda}(\psi))\rTo  H^1_{S}(K, \Lambda^\sharp(\Gamma)\otimes V_{L_\lambda}(\psi)) \rTo  \\
    H^1_{S}(\QQ, \Ind_{K}^\QQ \left(\Lambda^\sharp(\Gamma)\otimes V_{L_\lambda}(\psi)\right))
    \rTo^\cong H^1_{S}(\QQ, \Lambda^\sharp(\Gamma) \otimes V_{L_\lambda}(f)).
   \end{split}
  \]
  
  \begin{theorem}\label{cor:compat}
   Let $\gamma\in V_L(\psi)$ and write $\gamma'$ for its image in $V_L(f)$ under the map given by Lemma~\ref{lem:katoisomorphisms}(b). Then we have
   \[
    \Phi_{\psi,f}\left(\e_{\infty}(\gamma)\otimes(\zeta_{p^n})^{\otimes(-1)}\right)=\kato_{\gamma'}.
   \]
  \end{theorem}
  
  \begin{proof}
   This is \cite[(15.16.1)]{kato04}; it is immediate from a comparison the interpolating properties of the two zeta elements, since an element of $H^1_{\Iw}(\QQ_\infty / \QQ, V_{L_\lambda}(f))$ is uniquely determined by its images under the dual exponential maps at each finite level in the tower $\QQ_\infty / \QQ$.
  \end{proof}
  
  \begin{proposition}\label{prop:mackey}
   We have a commutative diagram
   \[ 
    \begin{diagram}
     H^1_{\Iw, S}(K_\infty, V_{L_\lambda}(\psi)) & \rTo^{\Phi_{\psi, f}} & H^1_{\Iw,S}(\QQ_\infty,V_{L_\lambda}(f))\\
     \dTo& \ldTo^\cong\\
     H^1_{\Iw, S}(K_\infty, V_{L_\lambda}(\psi) \oplus \iota V_{L_\lambda}(\psi))^{\alpha = 1} & 
    \end{diagram}
   \]
   where the left-hand vertical map sends $x$ to $x \oplus \delta \cdot \iota_*(x)$, and the diagonal isomorphism is given by restriction.
  \end{proposition}
  
  \begin{proof} 
   Clear. 
  \end{proof}

        
 \section{Critical-slope \texorpdfstring{$L$}{L}-functions}
  
  Let $f$ be a modular form of CM type, as above, and $\psi$ the corresponding Gr\"ossencharacter. We choose a basis $\gamma$ of $V_L(\psi)$, and let $\gamma'$ be its image in $V_L(f)$ under the isomorphism of Lemma \ref{lem:katoisomorphisms}(b).
  
  We fix an embedding $\overline{K} \into \Qpb$ which induces the $\lambda$-adic valuation on $L$. This gives an embedding $\Gal(\Qpb / \Qp) \into \Gal(\overline{K} / \QQ)$, whose image is contained in the subgroup $\Gal(\overline{K} / K)$. This gives a localization map
  \[ \loc_p: H^1_{\Iw, S}(\QQ_\infty, M) \rTo H^1_{\Iw}(\QQ_{p,\infty}, M)\]
  for each $\Gal(K^S / \QQ)$-module $M$. Moreover, we have a map
  \[ \loc_{\fp} : H^1_{\Iw, S}(K_\infty, M) \rTo H^1_{\Iw}(\QQ_{p,\infty}, M)\]
  for each $\Gal(K^S / K)$-module $M$, and we clearly have $\loc_{p} = \loc_{\fp} \circ \res_{K / \QQ}$.
  
  Via the isomorphism of Lemma \ref{lem:katoisomorphisms}(a), the space $V_{L_\lambda}(f)$ is isomorphic as a representation of $\Gal(\Qpb / \Qp)$ to $V_{L_\lambda}(\psi) \oplus \iota\left( V_{L_\lambda}(\psi)\right)$. Note that $\iota$ does not normalize the image of $\Gal(\Qpb / \Qp)$, so the two factors are non-isomorphic; indeed $V_{L_\lambda}(\psi)$ has Hodge--Tate weight $1-k$, while $\iota\left( V_{L_\lambda}(\psi)\right)$ has Hodge--Tate weight 0. Hence we have
  \[ \loc_p(\kato_{\gamma'}) \in H^1_{\Iw}(\QQ_{p,\infty}, V_{L_\lambda}(\psi)) \oplus H^1_{\Iw}(\QQ_{p,\infty}, \iota(V_{L_\lambda}(\psi))).\]
  
  Let us write $\pr_1$ and $\pr_2$ for the projections to the two direct summands above. By Corollary \ref{cor:compat}, the projection $\pr_1 \loc_p(\kato_{\gamma'})$ to $H^1_{\Iw}(\QQ_{p, \infty}, V_{L_\lambda}(\psi))$ is 
  \[ \loc_{\fp}\left(\e_\infty(\gamma) \otimes (\zeta_{p^n})^{\otimes (-1)}\right).\]
  By Proposition \ref{prop:mackey}, we see that the projection of $\loc_p(\kato_{\gamma'})$ to the other direct summand is 
  \[ \delta \cdot \loc_{\fp}\left[\iota_*\left(\e_\infty(\gamma) \otimes (\zeta_{p^n})^{\otimes (-1)}\right)\right] = \left[\delta \cdot \loc_{\fp}\left(\iota_*(\e_\infty(\gamma))\right)\right] \otimes (\zeta_{p^n})^{\otimes (-1)}.\]
  
  We have
  \[ \iota_* \left(\e_{\infty}(\gamma)\right) = \e_{\infty}(\iota \gamma),\]
  so this simplifies to
  \[ \pr_2\left(\loc_p \kato_{\gamma'}\right) = \delta \cdot \left[ \loc_{\fp} \left(\e_\infty(\iota \gamma)\right)\right] \otimes (\zeta_{p^n})^{\otimes(-1)}.\]
  
  \begin{definition}
   Let $L_{p, 1}^\gamma \in \Lambda(\Gamma) \otimes_{\Zp} \Dcris(V_{L_{\lambda}}(\psi)(k-1))$ and $L_{p, 2}^\gamma \in \Lambda(\Gamma) \otimes_{\Zp} \Dcris(\iota V_{L_{\lambda}}(\psi)(k-1))$ be the unique elements such that 
   \[ \mathcal{L}^\Gamma_{V_{L_\lambda}(f)(k-1)}\left(\kato_{\gamma'} \otimes (\zeta_{p^n})^{\otimes(k-1)}\right) = L_{p, 1}^\gamma \oplus L_{p, 2}^\gamma. \]
  \end{definition}

  We shall see below that if $g = \bar f$ is the complex conjugate of $f$, then $L_{p, 1}^\gamma$ will be the ordinary $p$-adic $L$-function of $g$, and $L_{p, 2}^\gamma$ is the critical-slope $p$-adic $L$-function of $g$.
  
  \begin{theorem}\label{thm:values}
   For every character $\eta$ of $\Gamma$, we have
   \[ L_{p, 1}^\gamma(\eta) = \mathbb{L}_{\ff p^\infty}(\eta \left(\psi_{\lambda} \chi^{k-2}\right)^{-1}) \cdot t^{k-1}\gamma,\]
   and
   \[ L_{p, 2}^\gamma(\eta) = \left(\ell_0 \dots \ell_{k-2} \delta \mathbb{L}_{\ff p^\infty}\right) (\eta \left(\psi_{\lambda}^\iota \chi^{k-2}\right)^{-1}) \cdot \iota \gamma.\]
  \end{theorem}

  \begin{proof}
   For brevity, we shall write $e_j$ for $(\zeta_{p^n})^{\otimes j}$, considered as a basis vector of $\Qp(j)$.

   It is easy to see that if $\xi$ is a character of $G_{\ff p^\infty}$ of the form $\chi^j \tau$, where $\tau$ is unramified and $j \ge 0$, and $V$ is any crystalline representation with non-negative Hodge-Tate weights, then for any $x \in H^1_{\Iw}(K(\ff p^\infty), V)$ and any choice of basis $e_\xi$ of $\Qp(\xi)$ we have
   \[  \mathcal{L}^{G_{\ff p^\infty}}_{V(\xi)}(x \otimes e_\xi)(\eta) = (\ell_0 \dots \ell_{j-1})(\eta) \cdot \mathcal{L}^{G_{\ff p^\infty}}_{V}(x)(\eta \xi^{-1}) \otimes t^{-j} e_{\xi}.\]
   Note that if $\xi$ takes values in the finite extension $L / \Qp$, this is an equality of two elements of $L \otimes \Finf \otimes \Dcris(V(\xi))$: the element $t^{-j} e_{\xi} \in \BB_{\cris} \otimes_{\Qp} L(\xi)$ transforms via $\tau$ under $G_{\Qp}$, and hence lies in $\Finf \otimes \Dcris(L(\xi))$, since the periods of unramified characters lie in $\Finf \subseteq \BB_{\cris}$.

   We apply this result with $V = \Qp$ (the trivial representation), $x = \e_{\ff p^\infty} \otimes e_{-1}$, and various values of $\xi$. Firstly, taking $\xi$ to be the cyclotomic character, we have
   \[ \mathbb{L}_{\ff p^\infty} = \ell_0^{-1} \mathcal{L}^{G_{\ff p^\infty}}_{\Qp(1)}(\e_{\ff p^\infty}),\]
   and thus
   \begin{equation}
    \label{eq:reg1}
    \mathbb{L}_{\ff p^\infty}(\eta) = \mathcal{L}^{G_{\ff p^\infty}}_{\Qp}(\e_{\ff p^\infty} \otimes e_{-1})(\chi^{-1} \eta) \otimes t^{-1} e_1.
   \end{equation}

   On the other hand we have 
   \begin{align*}
    L_{p, 1}^\gamma(\eta) &= \mathcal{L}^{\Gamma}_{V_{L_\lambda}(\psi)(k-1)}\left(\pr_1(\kato_{\gamma'}) \otimes e_{k-1})\right)(\eta) \\
    &= \mathcal{L}^{G_{\ff p^\infty}}_{V_{L_\lambda}(\psi)(k-1)}\left(\e_\infty(\gamma) \otimes e_{k-2}\right)(\eta)
   \end{align*}
   The group $G_{\Qp}$ acts on $V_{L_\lambda}(\psi)(k-1)$ via the unramified character $\chi^{k-1} \psi_{L_\lambda}$, so this is
   \begin{align*}
    L_{p, 1}^\gamma(\eta)=\mathcal{L}^{G_{\ff p^\infty}}_{\Qp}\left(\e_\infty \otimes e_{-1}\right)\left((\chi^{k-1} \psi_{L_\lambda})^{-1}\eta\right) \otimes (t^{k-1} \gamma) \otimes (t^{1-k} e_{k-1}).
   \end{align*}
   Comparing this with \eqref{eq:reg1}, we deduce that
   \[ L_{p, 1}^\gamma(\eta) = \mathbb{L}_{\ff p^\infty}\left((\chi^{k-2} \psi_{L_\lambda})^{-1}\eta\right) \otimes (t^{k-1} \gamma) \otimes (t^{2-k} e_{k-2}).\]
   If we identify $\Dcris(\Qp(k-2))$ with $\Qp$ in the usual way, $t^{2-k} e_{k-2}$ is sent to 1. As remarked above, the element $t^{k-1} \gamma \in \BB_{\cris} \otimes_{\Qp} V_{L_{\lambda}}(\psi)$ lies in $\Finf \otimes_{\Qp} \Dcris(V_{L_\lambda}(\psi))$. So if $\omega$ is a $K$-basis of $S(\psi)$, then the image of $\omega$ under the crystalline comparison isomorphism is a basis of $\DD_{\cris}(V_{L_\lambda}(\psi))$, and if we define $\Omega_p = (\gamma \otimes e_{1-k}) / \omega$, this will lie in $\Finf$ and our result becomes
   \[ L_{p, 1}^\gamma(\eta) = \mathbb{L}_{\ff p^\infty}\left((\chi^{k-2} \psi_{L_\lambda})^{-1}\eta\right) \cdot \Omega_p \omega.\]

   We now turn to $L_{p, 2}^\gamma$. We have
   \begin{align*}
    L_{p, 2}^\gamma(\eta) &= \mathcal{L}^{\Gamma}_{\iota(V_{L_\lambda}(\psi))(k-1)}\left( \pr_2(\kato_{\gamma'}) \otimes e_{k-1}\right)(\eta)\\
    &=  \mathcal{L}^{G_{\ff p^\infty}}_{\iota(V_{L_\lambda}(\psi))(k-1)}\left(\left(\delta \cdot \e_\infty(\iota \gamma)\right) \otimes e_{k-2} \right)(\eta)\\
    &= (-1)^{k-2} \eta(\delta) \mathcal{L}^{G_{\ff p^\infty}}_{\iota(V_{L_\lambda}(\psi))(k-1)}\left(\e_\infty(\iota \gamma) \otimes e_{k-2} \right)(\eta).
   \end{align*}
   The group $G_{\Qp}$ acts on $\iota(V_{L_\lambda}(\psi))$ by the character $\psi_{\lambda}^\iota$, which is unramified; so this is
   \begin{align*}
    L_{p, 2}^\gamma(\eta) &= (-1)^{k-2} \eta(\delta) (\ell_0 \dots \ell_{k-2})(\eta) \cdot  \mathcal{L}^{G_{\ff p^\infty}}_{\Qp}\left(\e_\infty \otimes e_{-1} \right)((\chi^{k-1} \psi_{\lambda}^\iota)^{-1} \eta) \\ &\hspace{3in} \otimes t^{1-k} e_{k-1} \otimes \iota\gamma.\\
    &= (-1)^{k-2} \eta(\delta) (\ell_0 \dots \ell_{k-2})(\eta) \cdot \mathbb{L}_{\ff p^\infty}\left((\chi^{k-2} \psi_{L_\lambda}^\iota)^{-1}\eta\right) \otimes t^{2-k} e_{k-2} \otimes \iota\gamma.
   \end{align*}
  As above, we identify $t^{2-k} e_{k-2} \in \Dcris(\Qp(k-2))$ with $1 \in \Qp$; and if $\omega$ is a basis of $S_L(\psi)$, the image of $\iota \omega$ under the comparison isomorphism is a basis of $\DD_{\cris}(\iota(V_{L_\lambda}(\psi)))$, so if we define $\Omega_p^\iota = (\iota \gamma) / (\iota \omega)$ this becomes
   \[ L_{p, 2}^\gamma(\eta) = (-1)^{k-2} \eta(\delta) (\ell_0 \dots \ell_{k-2})(\eta) \cdot \mathbb{L}_{\ff p^\infty}\left((\chi^{k-2} \psi_{L_\lambda})^{-1}\eta\right) \cdot \Omega_p^\iota \iota\omega.\]
  \end{proof}

  \begin{definition}
   Let $\omega$ be a basis of $S_L(\psi)$ as above, and let $L_{p, \alpha}(g)$ and $L_{p, \beta}(g)$ be the elements of $\calH_{L_\lambda}(\Gamma)$ defined by
   \[ L_{p, 1}^\gamma = L_{p, \alpha}(g) \cdot \omega\]
   and
   \[ L_{p, 2}^\gamma = L_{p, \beta}(g) \cdot \iota\omega.\]
   Then $L_{p,\alpha}$ and $L_{p, \beta}$ are the $p$-adic $L$-functions attached to $g$, where $\alpha$ and $\beta$ are respectively the unit and non-unit roots of the Hecke polynomial of $g$.
  \end{definition}

  As shown in \cite[\S 16]{kato04}, this is consistent with the classical Amice--Velu--Vishik construction of the ordinary $p$-adic $L$-function $L_{p,\alpha}(g)$, and thus it is natural to regard $L_{p, \beta}(g)$ as a candidate for a critical-slope $p$-adic $L$-function. This is the definition of the Kato critical-slope $L$-function used in \cite{loefflerzerbes10}.

  \begin{theorem}\label{thm:maintheorem}
   Up to multiplication by two nonzero scalars, one for each sign, $L_{p, \beta}(g)$ coincides with the modular symbol critical-slope $L$-function $L_{p, \beta}^{\mathrm{MS}}(g)$ attached to the non-ordinary $p$-stabilization of $f$ in \cite{bellaiche11b}.
  \end{theorem}

  \begin{proof}
   This follows by comparing the formulae of Theorem \ref{thm:values} with Theorem 2 of \cite{bellaiche11b}. Note that Bella\"iche shows that if $\rho_1$ and $\rho_2$ are the two characters by which $\Gal(\overline{K} / K)$ acts on $V_g^*$, then
   \[\begin{cases}
    L_{p, \alpha}(g)(\eta) &= \mathbb{L}_{\ff p^\infty}(\rho_2 \eta^{-1}) \cdot (\text{constant}^\pm),\\
    L_{p, \beta}^{\mathrm{MS}}(g)(\eta) &= (\ell_0 \cdots \ell_{k-2})(\eta) \cdot \mathbb{L}_{\ff p^\infty}(\rho_1 \eta^{-1}) \cdot (\text{constant}^\pm).
   \end{cases}\]
   Here $\text{constant}^\pm$ indicates an equality of distributions on $\Gamma$ up to multiplication by two nonzero constants (one for each sign). On the other hand, we have proved that
   \[\begin{cases}
    L_{p, \alpha}(g)(\eta) &= \mathbb{L}_{\ff p^\infty}(\chi \rho_1^{-1} \eta) \cdot (\text{constant}),\\
    L_{p, \beta}(g)(\eta) &= (\ell_0 \cdots \ell_{k-2})(\eta) \cdot \mathbb{L}_{\ff p^\infty}(\chi \rho_2^{-1} \eta) \cdot (\text{constant}).
   \end{cases}\]
   To reconcile these formulae, we note that the $p$-adic $L$-function $\mathbb{L}_{\ff p^\infty}$ satisfies a functional equation \cite[\S II.6]{deshalit87}
   \[ \mathbb{L}_{\ff p^\infty}(\iota(\eta)) = C(\eta) \cdot \mathbb{L}_{\ff p^\infty}(\chi \eta^{-1}),\]
   for a function $C(\eta)$ (involving a $p$-adic root number and various other correction terms) which depends only on the coset of $\eta$ modulo characters factoring through $\Gal(\QQ_{\infty}^+ / \QQ)$. Since $\iota(\rho_1) = \rho_2$ and vice versa, we deduce that 
   \[ L_{p, \beta}(g) = L_{p, \beta}^{\mathrm{MS}}(g) \cdot (\text{constant}^\pm).\]
   Since the modular symbol $L$-function is only defined up to scalars, this completes the proof.
  \end{proof}
 
  \begin{note}
   Both Kato's and Bella\"iche's critical-slope $p$-adic $L$-functions are only defined up to multiplication by a nonzero constant for characters of each sign; in Kato's construction these constants correspond to the choice of $\gamma$, whose projection to each of the $\pm$ eigenspaces of complex conjugation must be non-zero. It seems natural to ask whether one can choose normalizations for both in a compatible fashion so Theorem \ref{thm:maintheorem} holds exactly, but the present authors do not feel sufficiently familiar with the modular symbol construction to comment further.
  \end{note}

\renewcommand{\MR}[1]{%
  MR \href{http://www.ams.org/mathscinet-getitem?mr=#1}{#1}.
}
\providecommand{\bysame}{\leavevmode\hbox to3em{\hrulefill}\thinspace}
\providecommand{\MRhref}[2]{%
  \href{http://www.ams.org/mathscinet-getitem?mr=#1}{#2}
}

\end{document}